\documentclass{article}
\usepackage{enumerate}
\usepackage{amsmath,amssymb,amsthm}
\usepackage{graphicx}

\usepackage[dvipsnames]{xcolor}
\usepackage{hyperref}
\hypersetup{
    colorlinks=true,
    linkcolor=cyan!80!black,
    citecolor=MidnightBlue,
    urlcolor=magenta,
}

\theoremstyle{plain}
\newtheorem{theorem}{Theorem}

\newtheorem{lemma}{Lemma}

\theoremstyle{definition}

\newtheorem{definition}{Definition}

\let\tilde\widetilde

\title{A short proof of free energy limit of two spin spherical Sherrington-Kirkpatrick model at any temperature}
\author{Debapratim Banerjee\\
    Department of Mathematics \\
    Ashoka University, India\\
    Email id: debapratim.banerjee@ashoka.edu.in
}
\begin{document}

\maketitle
\begin{abstract}
    In this paper, we present a short proof of the limit of free energy of spherical $2$ spin Sherrington-Kirkpatrick (SSK) model without external field. This proof works for all temperatures and is based on the Laplace method of integration and considering a discrete approximation of the whole space. This proof is general enough and can be adapted to any interaction matrix where the eigenvalue distribution has some nice properties. The proof in the high-temperature case is the same as the proof given in \cite{belius2019tap}. However, the low-temperature case is new.
\end{abstract}
\section{Introduction}
The Sherrington- Kirkpatrick model (SK model) is an important model in mathematical physics originally proposed by Sherrington and Kirkpatrick \cite{sherrington1975solvable}. This model undergoes a phase transition with respect to the temperature parameter. Earlier, rigorous studies dealt with the high temperature phase of the model. One might look at \cite{ALR87} for example. In the physics literature, there were two parallel approaches to work with the model. The first approach is replica based and is known as the Parisi approach \cite{mezard1987spin} and the other approach is by Thouless, Anderson and Palmar and is known as the TAP approach \cite{thouless1977solution}. Perhaps, the most significant break through in this literature is by Talagrand where he rigorously proved the Parisi formula in his seminal paper \cite{Tal05}. Based on this paper, a huge literature was built by several other authors. On the other hand, to the best of my knowledge, a direct proof of the TAP approach for Ising spins is still not available in the literature. This project started in early 2024 when the author tried to find a direct approach to the TAP free energy by diagonalizing the Hamiltonian. Here, he came with a conjecture about the free energy at low temperature for both Ising and Spherical cases. This note was shared with experts through private communications. In order to work with this new approach, we came across this easy proof of free energy of Spherical SK model. Spherical SK model without an external fields is a solved problem. Many papers contain the result of this paper. For example, one might look at \cite{talagrand2006free}, \cite{panchenko2007overlap}, \cite{belius2019tap}, \cite{baik2017fluctuations} etc. However, the proof presented in this paper is quite short and neat. This proof answers an earlier question asked by Belius through a private communication which is a key step to prove the TAP type free energy the author proposed. The current proof in this paper works for all temperatures. As the first draft of this paper was completed, it was pointed out by Belius that the proof in the high-temperature regime is same as the proof given in \cite{belius2019tap}. However, the proof in the low-temperature phase is original. The key idea of the proof is as follows: After diagonalizing the matrix, we look at the overlap of the spin vectors with the eigenvectors of the interaction matrix. We see the whole problem as a Laplace integration problem and the partition the space into several disjoint parts. We consider the case when higher eigenvalues have a non-trivial overlap in the corresponding eigenvectors and the smaller eigenvalues will correspond to a high-temperature calculation. We use large deviation theory for the higher eigenvalues and optimize over the whole space. Since the eigenvector overlaps with the spins are spherically symmetric for the spherical symmetric case, the calculation becomes easy. In particular, we show that for an inverse temperature $\beta >1$, the combined overlap at the top few $o(n)$ eigenvalues is $\frac{\beta-1}{\beta}$. We have a conjecture for the ISING case as well. The calculation in this paper gives a rough estimate which is enough to calculate the free energy limits in the spherical case. However, a detailed analysis of the high-temperature part should give some deeper insights about the free energy like the fluctuation. This method gives an alternative and direct way to look at the free energy of the SK model.


\section{Formal definition of the model}
In this section, we give the formal definition of the model.
\begin{definition}(GOE matrix)
    Let $X_{n}= \frac{1}{\sqrt{n}}\left(  x_{i,j} \right)_{1\le i \le j \le n}$ be a symmetric matrix of dimension $n \times n$. It is called a GOE matrix (Gaussian orthogonal ensemble) if $x_{i,j}\sim N(0,1)$ for $1\le i < j \le n$, $x_{i,i} \sim N(0,2)$ for $1 \le i \le n$ and the random variables are independent.
\end{definition}
Let $\lambda_{1}\ge \ldots \ge \lambda_{n}$ be the eigenvalues of the matrix $X_{n}$. It is a classical result in random matrix theory that 
\[
\frac{1}{n}\sum_{i=1} \delta_{\lambda_{i}}
\]
converges weakly to the semicircular distribution in the almost sure sense. We now formally define the semicircular law.
\begin{definition}\label{def:sc}
    The semicircular law is a probability distribution on $[-2,2]$ and it's probability density function is 
    \begin{equation}
        f(x)= \frac{1}{2\pi} \sqrt{4-x^2}~ \mathbb{I}_{|x|\le 2}.
    \end{equation}
    In this paper, we denote the probability measure of the semicircular law by $\mu_{\mathrm{sc}}$.
\end{definition}
Another important concept we need is the Stieltjes transform of a probability distribution.
\begin{definition}(Stieltjes transform)
    Let $\mu$ be a probability distribution supported on  $I \subset \mathbb{R}$. Then the Stieltjes transform of $\mu$ at a (possibly complex value) $z \in \mathbb{C} \backslash I$ is defined as 
    \[
     S_{\mu}(z) = \int_{I} \frac{1}{x-z} d \mu(x).
    \]
    In particular, when $\mu$ is the semicircular law, we have 
    \[
    S_{\mu}(z) = \frac{-z + \sqrt{z^2-4}}{2}
    \]
    for $z \in \mathbb{C} \backslash [-2,2]$.
\end{definition}
\begin{definition}(Uniform distribution on the sphere)
For any $n \ge 2$, define $\mathbb{S}_{n-1}$ as the set of all $\underline{\sigma}=(\sigma_{1},\ldots, \sigma_{n})' \in \mathbb{R}^{n}$ such that $||\underline{\sigma}||^{2}=1$. $\mathbb{S}_{n-1}$ is called the unit sphere in $\mathbb{R}^{n}$. Let $\mu_{n}$ be the uniform probability measure on $\mathbb{S}_{n-1}$.
\end{definition}
\begin{definition}(Hamiltonian of SK model)
    Let $X_{n}$ be a G.O.E. matrix of dimension $n \times n$. For any $\underline{\sigma} \in \mathbb{S}_{n-1}$, the Hamiltonian of the Sherrington-Kirkpatrick model is defined as 
    \[
    H_{n}(\underline{\sigma})= \frac{n}{2} \underline{\sigma}' X_{n} \underline{\sigma}.
    \]
\end{definition}
\begin{definition}(Partition function and free energy)\label{def:partitionfree}
    The partition function of the spherical-SK model at an inverse temperature $\beta >0$, denoted as $Z_{n}(\beta)$, is defined as follows:
    \[
      Z_{n}(\beta)= \int_{\mathbb{S}_{n-1}} \exp\left\{ \beta H_{n}(\underline{\sigma}) \right\} d\mu_{n}(\underline{\sigma}).
    \]
    The Free energy, $F_{n}(\beta)$ is defined as $\log (Z_{n}(\beta))$.
\end{definition}
Free energy is a component of central interest in the study of spin-glass models and this paper is about the limit of $\frac{F_{n}(\beta)}{n}$ as $n \to \infty$. 
\section{Main result}
Now, we state the main result of this paper: 
\begin{theorem}\label{thm:main}(Free energy limit at any temperature)
Let $Z_{n}(\beta)$ and $F_{n}(\beta)$ be as defined in Definition \ref{def:partitionfree}. Then 
\begin{equation}
    \lim_{n \to \infty} \frac{F_{n}(\beta)}{n} = \left\{
     \begin{array}{cc}
         \frac{\beta^2}{4} & \text{whenever } \beta \le 1 \\
          \beta - \frac{3}{4}- \frac{1}{2}\log(\beta)& \text{whenever } \beta >1
     \end{array}
    \right.
\end{equation}
in the almost sure sense.
\end{theorem}
\section{Dirichlet distributions}
The proof of Theorem \ref{thm:main} requires some properties of Dirichlet distributions. So we at first introduce them.   
\begin{definition}(Dirichlet distributions)
Dirichlet distributions are a family of multivariate distributions parametrized by a vector $\underline{\alpha}=(\alpha_{1},\ldots, \alpha_{k})$. Suppose $\underline{Y}=(Y_{1}, \ldots, Y_{k}) \sim \mathrm{Dir}(\alpha_{1},\ldots, \alpha_{k})$, then $Y$ is supported on the simplex $S_{k}:=\{0\le y_{i} \le 1\} \cap \{ \sum_{i=1}^{k} y_{i} =1\}$. The p.d.f. (probability density function) at $(y_{1},\ldots , y_{k})$ is given by 
\begin{equation}\label{eq:dirden}
    f_{\alpha_{1},\ldots, \alpha_{k}}(y_{1},\ldots, y_{k})= \frac{\Gamma(\alpha_{1}+ \ldots + \alpha_{k})}{\prod_{i=1}^{k} \Gamma(\alpha_{i})} \prod_{i=1}^{k} y_{i}^{\alpha_{i}-1}.
\end{equation}
\end{definition}
We would require the following property of the Dirichlet distribution: 
\begin{lemma}\label{lem:diruse}
Suppose $X_{1},\ldots, X_{k}$ are independent $\Gamma(\alpha_{1},\tau),\ldots \Gamma(\alpha_{k},\tau)$ respectively, then $(Y_{1},\ldots, Y_{k}) \sim \mathrm{Dir}(\alpha_{1},\ldots, \alpha_{k})$, where $Y_{i}:=  \frac{X_{i}}{X_{1}+ \ldots + X_{k}}$.
\end{lemma}
The proof of Lemma \ref{lem:diruse} is simple and uses a multivariate change of variable formula. So we omit it.
We know that $\chi^{2}(2\alpha)$ distribution is same as  $\Gamma(\alpha,2)$ distribution. Hence, in Lemma \ref{lem:diruse}, we can take $X_{1},\ldots, X_{k}$ independent $\chi^{2}(2\alpha_{1}), \ldots \chi^{2}(2\alpha_{k})$.
\section{Proof of the main result}
In this section, we provide the proof of Theorem \ref{thm:main}. 
\begin{proof}[Proof of Theorem \ref{thm:main}]
The proof of Theorem \ref{thm:main} contains the following four parts: (i) Diagonalization, (ii) Discretization (iii) Optimization  and (iv) Laplace integration method. We shall complete these steps one by one. 

\noindent 
\textbf{ (i) Diagonalization:} This step is straightforward. As $\underline{\sigma}$ is uniformly distributed over $\mathbb{S}_{n-1}$, $H\underline{\sigma}$ is also uniformly distributed over $\mathbb{S}_{n-1}$ for any orthogonal matrix $H$. As a consequence, 
\begin{equation}
\begin{split}
Z_{n}(\beta) &= \int_{\mathbb{S}_{n-1}} \exp\left\{ \beta H_{n}(\underline{\sigma}) \right\} d\mu_{n}(\underline{\sigma})\\
&=\int_{\mathbb{S}_{n-1}} \exp\left\{ \frac{n\beta}{2} \underline{\sigma}' H_{n}' \Lambda_{n} H_{n} \underline{\sigma} \right\} d\mu_{n}(\underline{\sigma})\\
&= \int_{\mathbb{S}_{n-1}} \exp\left\{ \frac{n\beta}{2} \sum_{i=1}^{n} \lambda_{i} \sigma_{i}^2 \right\} d\mu_{n}(\underline{\sigma}).
\end{split}
\end{equation}
Here $X_{n} = H_{n}' \Lambda H_{n}$ is the spectral decomposition of $X_{n}$ and $\lambda_{1}\ge \lambda_{2}\ge \ldots \ge \lambda_{n}$ are the eigenvalues of $X_{n}$.

\noindent 
\textbf{(ii) Discretization:}
We take a finite but large enough $K$ and make a discrete approximation of the semicircular law. Let $2=\tilde{\lambda}_{1}\ge \ldots \ge \tilde{\lambda}_{K}\ge \tilde{\lambda}_{K+1}=-2$ be such that 
\[
\int_{\tilde{\lambda}_{j+1}}^{\tilde{\lambda}_{j}} d\mu_{\mathrm{sc}}(x)= \frac{1}{K}. 
\]
Fixing $\varepsilon >0$, we take $K$ large enough so that 
\begin{equation}\label{eq:appI}
 \sup_{1 \le i \le K} \sup_{\frac{(i-1)n}{K}+1 \le j \le \frac{in}{K}} \left| \lambda_{j}- \tilde{\lambda}_{i} \right| \le \varepsilon.
\end{equation}
This is ensured from the eigenvalue rigidity of the GOE matrix. For example see \cite{erdHos2012rigidity}. From \eqref{eq:appI}, observe that 
\begin{equation}\label{eq:appII}
    \begin{split}
      \left| \frac{n\beta}{2} \sum_{j=1}^{n} \lambda_{j} \sigma_{j}^2 - \frac{n\beta}{2} \sum_{i=1}^{K} \tilde{\lambda}_{i}\left(\sum_{j=\frac{(i-1)n}{K}+1}^{\frac{in}{K}}\sigma_{j}^2\right)  \right| \le \frac{n\beta \varepsilon}{2}.
    \end{split}
\end{equation}
Hence, we shall work with 
\begin{equation}\label{eq:partitionup}
\int_{\mathbb{S}_{n-1}} \exp\left\{ \frac{n\beta}{2} \sum_{i=1}^{K} \tilde{\lambda}_{i}\left(\sum_{j=\frac{(i-1)n}{K}+1}^{\frac{in}{K}}\sigma_{j}^2\right) \right\} d\mu_{n}(\underline{\sigma})
\end{equation}
instead of the actual partition function.  Recall that $\left(\sigma_{1},\ldots, \sigma_{n}\right)' $ has the same distribution as $\left(\frac{Y_{1}}{||Y||},\ldots, \frac{Y_{n}}{||Y||}\right)'$ where $Y= (Y_{1},\ldots, Y_{n})' \sim N_{n}(0,I_{n})$. As a consequence the vector 
$\left( \sum_{j=\frac{(i-1)n}{K}+1}^{\frac{in}{K}}\sigma_{j}^2\right)_{1\le i \le K} \sim \mathrm{Dir}\left( \frac{n}{2K},\ldots, \frac{n}{2K} \right)$. Hence, \eqref{eq:partitionup} can be rewritten as 
\[
\mathrm{E}\left[ \exp\left\{ \frac{n\beta}{2}\sum_{i=1}^{K} \tilde{\lambda }_{i} r_{i}\right\} \right]
\]
where $(r_{1},\ldots, r_{K}) \sim \mathrm{Dir}(\frac{n}{2K},\ldots, \frac{n}{2K})$.
Hence, 
\begin{equation}\label{eq:integrand}
\mathrm{E}\left[ \exp\left\{ \frac{n\beta}{2}\sum_{i=1}^{K} \tilde{\lambda }_{i} r_{i}\right\} \right]= \int_{S_{K}} \exp\left\{ \frac{n\beta}{2}\sum_{i=1}^{K} \tilde{\lambda }_{i} r_{i}\right\} f_{\frac{n}{2K},\ldots, \frac{n}{2K}}(r_{1},\ldots,r_{K}) dr_{1}\ldots dr_{K}
\end{equation}

This concludes our second step.\\
\textbf{(iii) Optimization:} We now come to the third step. First of all, we do some asymptotics on the p.d.f. of $\mathrm{Dir}\left( \frac{n}{2K}, \ldots, \frac{n}{2K} \right)$. From \eqref{eq:dirden} we know that 
\begin{equation}
    \begin{split}
      f_{\frac{n}{2K},\ldots, \frac{n}{2K}}(v_{1},\ldots, v_{K})= \frac{\Gamma(\frac{n}{2})}{\left( \Gamma(\frac{n}{2K}) \right)^{K}} \prod_{i=1}^{k} v_{i}^{\frac{n}{2K}-1}.
    \end{split}
\end{equation}
We approximate $\Gamma(\frac{n}{2})$ by $\left( \frac{n}{2}
\right)^{\frac{n}{2}}\exp\left\{ - \frac{n}{2} +o(n) \right\}$ and $\Gamma(\frac{n}{2K})$ by $\left( \frac{n}{2K} \right)^{\frac{n}{2K}}\exp\left\{ - \frac{n}{2K} +o(n)\right\}$. These are estimates are obtained from Stirling's approximation. Hence, 
\begin{equation}
\begin{split}
\frac{\Gamma(\frac{n}{2})}{\left( \Gamma(\frac{n}{2K}) \right)^{K}} &= \left(K\right)^{\frac{n}{2}} \exp\{ o(n)\}\\
&= \exp\left\{ \frac{n}{2}\log K + o(n) \right\}.
\end{split}
\end{equation}
We shall now consider the integrand in the r.h.s. of \eqref{eq:integrand} and shall optimize with respect to $v_{1},\ldots, v_{K}$. The integrand in the r.h.s. of \eqref{eq:integrand} becomes: 
\begin{equation}\label{eq:optimizing}
    \begin{split}
     \exp\left\{ \frac{n\beta}{2} \sum_{i=1}^{K} \tilde{\lambda}_{i} v_{i} + \frac{n\log K}{2} + \left(\frac{n}{2K}-1\right) \sum_{i=1}^{K} \log v_{i}  \right\}
    \end{split}
\end{equation}
The rest of this step is devoted to maximize \eqref{eq:optimizing} where $v_{1},\ldots, v_{K}$ varies over $S_{K}$. Hence, it is enough to find
\begin{equation}\label{eq:optact}
    \max_{v_{1},\ldots,v_{K} \in S_{K}} \frac{n\beta}{2} \sum_{i=1}^{K} \tilde{\lambda}_{i} v_{i} + \frac{n\log K}{2} + \left(\frac{n}{2K}-1\right) \sum_{i=1}^{K} \log v_{i}.
\end{equation}
\eqref{eq:optact}
is concave in $v_{1},\ldots, v_{K}$. Hence, the Lagrange multiplier method gives the maximum value when $v_{1},\ldots, v_{K}$ varies over $S_{K}$. 
Consider the following function:
\begin{equation}
\begin{split}
    G(v_{1},\ldots,v_{K}, \Gamma):=
     \frac{\beta}{2} \sum_{i=1}^{K} \tilde{\lambda}_{i} v_{i} + \frac{\log K}{2} + \left(\frac{1}{2K}-\frac{1}{n}\right) \sum_{i=1}^{K} \log v_{i} - \Gamma\left(  \sum_{i=1}^{K} v_{i}-1\right).
\end{split}    
\end{equation}

\noindent 
Taking $ \frac{\partial G}{\partial v_{i}}$ and putting it $0$, we have 
\begin{equation}
\begin{split}
    &\frac{\partial G}{\partial v_{i}}= \frac{\beta}{2} \tilde{\lambda}_{i} + \left( \frac{1}{2K} - \frac{1}{n}\right) \frac{1}{v_{i}} - \Gamma = 0\\ 
    & \Rightarrow \left( \frac{1}{2K} - \frac{1}{n}\right) \frac{1}{v_{i}}= \Gamma - \frac{\beta}{2} \tilde{\lambda}_{i}\\
    & \Rightarrow \left( \frac{1}{2K} - \frac{1}{n}\right) \frac{1}{\Gamma - \frac{\beta}{2} \tilde{\lambda}_{i}} = v_{i}
\end{split}    
\end{equation}
The fact that $\sum_{i=1}^{K} v_{i}=1$ gives:
\begin{equation}\label{eq:resolvent}
    \begin{split}
      \left( \frac{1}{2K} - \frac{1}{n}\right) \sum_{i=1}^{K} \frac{1}{\Gamma - \frac{\beta}{2}\tilde{\lambda}_{i}} =1 \\ \Rightarrow \left( \frac{1}{K}- \frac{2}{n} \right) \sum_{i=1}^{K} \frac{1}{\frac{2\Gamma}{\beta }- \tilde{\lambda}_{i}}= \beta
    \end{split}
\end{equation}
where $\frac{2\Gamma}{\beta} > \tilde{\lambda}_{1}=2$. Observe that as $K$ becomes larger and larger, the l.h.s. of second display of \eqref{eq:resolvent} converges to $-S_{\mu_{\mathrm{sc}}}(\frac{2\Gamma}{\beta})$.
We now analyze \eqref{eq:resolvent} case by case.

\noindent 
\textbf{Case I:} Here we consider $\beta <1$. We know that $-S_{\mu_{\mathrm{sc}}}(\beta + \frac{1}{\beta})= \beta$ whenever $\beta < 1 $. Hence, putting $\frac{2\Gamma}{\beta}= \beta+ \frac{1}{\beta}$ makes the l.h.s. of second display of \eqref{eq:resolvent} $\beta + o_{K}(1)$. Using this value of $\frac{2\Gamma}{\beta}$, we get $v_{i}= \frac{1}{\beta K}\frac{1}{\frac{2\Gamma}{\beta}-\tilde{\lambda}_{i}} +o_{K}(1).$ This is the optimizing value of $v_{i}$.
Hence, at the optimizer, the value of \eqref{eq:optact} becomes:
\begin{equation}\label{eq:hightemp}
    \begin{split}
      &\frac{n}{2K} \sum_{i=1}^{K} \frac{\tilde{\lambda}_{i}}{\beta+ \frac{1}{\beta}- \tilde{\lambda}_{i}}\\
      & + \left(\frac{n}{2K}\right) \sum_{i=1}^{K} \left( -\log(\beta) - \log(K) - \log\left(\beta+ \frac{1}{\beta}- \tilde{\lambda}_{i}\right) \right) + \frac{n\log K}{2} + o(n)\\
      =& -\frac{n}{2} + \frac{n(\beta+\frac{1}{\beta})}{2K} \sum_{i=1}^{K} \frac{1}{(\beta+ \frac{1}{\beta})-\tilde{\lambda}_{i}} - \frac{n\log (\beta) }{2} - \frac{n}{2K} \sum_{i=1}^{K} \log \left( \beta+ \frac{1}{\beta} -  \tilde{\lambda}_{i}\right) + o(n)\\
      =& -\frac{n}{2} + \frac{n(\beta + \frac{1}{\beta})}{2} \times \beta - \frac{n\log (\beta)}{2}- \frac{n}{2} \int_{-2}^{2} \log \left( \beta+ \frac{1}{\beta} - x \right)d\mu_{\mathrm{sc}}(x) + o(n)\\
      &= \frac{n\beta^2}{2}- \frac{n\log(\beta)}{2}- \frac{n}{2} \int_{-2}^{2} \log \left( \beta+ \frac{1}{\beta} - x \right)d\mu_{\mathrm{sc}}(x) + o(n)
    \end{split}
\end{equation}
where the $o(n)$ terms are less $\varepsilon n$ for any $\varepsilon >0$ for large enough $K$.
It is a well known fact that 
\begin{equation}\label{eq:logint}
\int_{-2}^{2} \log (z-x)d\mu_{\mathrm{sc}}(x)= \frac{1}{4}z\left(z- \sqrt{z^2-4}\right) + \log \left( z+ \sqrt{z^2-4} \right)- \log 2 -\frac{1}{2}
\end{equation}
for $z\ge 2$ \cite{baik2017fluctuations}. Putting $z= \beta + \frac{1}{\beta}$, \eqref{eq:logint} becomes:
\begin{equation}
\begin{split}
    &\frac{1}{4}\left(\beta + \frac{1}{\beta}\right)(2\beta) + \log\left(\frac{2}{\beta}\right) - \log(2)- \frac{1}{2}\\
    & =\frac{\beta^2}{2} - \log(\beta). 
\end{split}    
\end{equation}
Putting this in \eqref{eq:hightemp}, we have \eqref{eq:hightemp} is $\frac{n\beta^2}{4}+o(n)$.

\noindent 
\textbf{Case II:} Now, we consider $\beta =1$. The analysis is essentially similar to the $\beta <1$ case as 
\[
\int_{-2}^{2} \frac{1}{2-x} d\mu_{\mathrm{sc}}(x)=1.
\] 
However, we cannot take $\frac{2\Gamma}{\beta}=2$ as $\tilde{\lambda}_{1}=2$. Here we fix $\varepsilon >0$ arbitrarily small, take $\frac{2\Gamma}{\beta}=2+ \varepsilon$ and repeat the analysis of the earlier case. 

\noindent 
\textbf{Case III:} Finally we consider the low temperature case, that is, $\beta >1$. Recall that it is not possible to find $z\ge 2$ such that $\int_{-2}^{2} \frac{1}{z-x}d\mu_{\mathrm{sc}}(x)= \beta.$ However, we actually need to find the solution of \eqref{eq:resolvent} rather than finding $z$ such that $\int_{-2}^{2} \frac{1}{z-x}d\mu_{\mathrm{sc}}(x)= \beta.$ In this case take $\frac{2\Gamma}{\beta}= 2+ \frac{1}{(\beta -1)K}$ so that $\frac{1}{K} \frac{1}{\frac{2\Gamma}{\beta}- 2}=\beta -1$. Now, we shall prove that for this choice of $\Gamma$,
\[
\frac{1}{K}\sum_{i=2}^{K} \frac{1}{\frac{2\Gamma}{\beta}- \tilde{\lambda}_{i}} = 1 +o(1) 
\]
where the $o(1)$ term goes to $0$ as $K \to \infty$. The proof is simple. We know that 
\[
\frac{1}{K}\sum_{i=2}^{K} \frac{1}{2- \tilde{\lambda}_{i}} \to \int_{-2}^{2} \frac{1}{2-x} d\mu_{\mathrm{sc}}(x)=1 
\]
as $K \to \infty$. Now ,
\begin{equation}
\begin{split}
&1=\lim_{\varepsilon \to 0} \lim_{K \to \infty}\frac{1}{K}\sum_{i=2}^{K} \frac{1}{2+\varepsilon- \tilde{\lambda}_{i}}\le \liminf_{K \to \infty}\frac{1}{K}\sum_{i=2}^{K} \frac{1}{\frac{2\Gamma}{\beta}- \tilde{\lambda}_{i}}\\
&~~~~~~~~~~~~~~~~~~~~~~~~~~~~~\le\limsup_{K \to \infty}\frac{1}{K}\sum_{i=2}^{K} \frac{1}{\frac{2\Gamma}{\beta}- \tilde{\lambda}_{i}} \le \lim_{K \to \infty} \frac{1}{K}\sum_{i=2}^{K} \frac{1}{2- \tilde{\lambda}_{i}}=1
\end{split}
\end{equation}
Hence, taking $\frac{2\Gamma}{\beta}= 2+ \frac{1}{(\beta-1)K}$, gives us 
\[
\frac{1}{K}\sum_{i=1}^{K} \frac{1}{\frac{2\Gamma}{\beta}-\tilde{\lambda}_{i}}= \beta +o(1).
\]
Hence, at the optimizer \eqref{eq:optact} becomes:
\begin{equation}
    \begin{split}
        &\frac{n\beta}{2} \sum_{i=1}^{K} \tilde{\lambda}_{i}v_{i} + \frac{n\log K}{2}+ \frac{n}{2K} \sum_{i=1}^{K} \log (v_{i}) +o(n) \\
        &= n\beta v_{1}+ \frac{n\beta}{2}\sum_{i=2}^{K}\tilde{\lambda}_{i} v_{i} + \frac{n\log K}{2} + \frac{n}{2K} \log(v_{1}) + \frac{n}{2K}\sum_{i=2}^{K} \log (v_{i}) +o(n)\\
        &= n \beta \frac{\beta-1}{\beta} + \frac{n}{2K}\sum_{i=2}^{K}\frac{\tilde{\lambda}_{i}}{\frac{2\Gamma}{\beta}- \tilde{\lambda}_{i}} + \frac{n\log K}{2} + \frac{n}{2K} \log\left(\frac{\beta-1}{\beta}\right) \\
        &~~~~~~~~~~~~~~~+ \frac{n}{2K} \sum_{i=2}^{K} \left( -\log K - \log \beta- \log\left(\frac{2\Gamma}{\beta}-\tilde{\lambda}_{i}\right) \right) +o(n)\\
        &= n(\beta-1) - \frac{n}{2} + n + \frac{n\log K}{2K} + \frac{n}{2K} \log\left( \frac{\beta-1}{\beta} \right)- \frac{n(K-1)}{2K}\log \beta\\
        &~~~~~~~~~~~~~~~~~~~~- \frac{n}{2} \int_{-2}^{2} \log(2-x) d\mu_{\mathrm{sc}}(x) +o(n).\\
        & =n(\beta -1) + \frac{n}{2} - \frac{n}{2}\log \beta - \frac{n}{2}\left( \frac{1}{2} \right) +o(n)\\
        &= n\beta -  \frac{3n}{4} - \frac{n}{2}\log (\beta) +o(n). 
    \end{split}
\end{equation}
This concludes the optimization step. 
We call the optimized value of \eqref{eq:optact} at any $\beta$, $F_{\mathrm{opt}}(\beta)$.

\noindent 
\textbf{(iv) Laplace's method of integration:} This step is standard. However, for completeness, we sketch it here.
Firstly, the l.h.s. of \eqref{eq:integrand} is clearly bounded by $\exp\left\{ F_{\mathrm{opt}}(\beta) \right\}$. On the other hand, the quantity \eqref{eq:optact} is continuous in it's arguments and the solution is in the interior of $S_{K}$. So for any $\varepsilon >0$, we can find an open ball $B_{\varepsilon}$ around the optimizer so that the value of \eqref{eq:optact} is uniformly bigger than $F_{opt} - n\varepsilon$ on $B_{\varepsilon}$. So l.h.s. of \eqref{eq:integrand} is bigger than $\exp\left\{  F_{opt} - n\varepsilon + \log (\mathrm{Vol}(B_{\varepsilon}))\right\}$. As $\log (\mathrm{Vol}(B_{\varepsilon}))$ is a finite quantity for every fixed $\varepsilon$, we get the result.
\end{proof}
\textbf{Acknowledgements:} I am grateful to Prof. David Belius and his group and Prof. Partha Sarathi Dey for private communications and discussions. 
\bibliographystyle{alpha}
\bibliography{main}
\end{document}